\documentclass[12pt]{amsart}
\usepackage{cite}
\usepackage{graphicx}
\usepackage{amsmath, amsfonts, amssymb,latexsym, graphics, graphicx, mathrsfs, manfnt, textcomp, appendix, chemarrow}
\usepackage{hyperref}
\usepackage[all]{xy}

\author{Ersin K\i zgut}
\address{Ersin K\i zgut, Department of Mathematics \\ Middle East Technical University \\ 06800 Ankara Turkey}
\email{kizgut@metu.edu.tr}
\author{Murat Yurdakul}
\address{Murat Yurdakul, Department of Mathematics \\ Middle East Technical University \\ 06800 Ankara Turkey}
\email{myur@metu.edu.tr}

\title[Factorized unbounded operators between Fr{\'e}chet spaces]{On the existence of a factorized unbounded operator between Fr{\'e}chet spaces}

\thanks{This research was partially supported by Turkish Scientific and Technical Research Council.}

\numberwithin{equation}{section}

\theoremstyle{thmit} 
\newtheorem{theorem}{Theorem}[section]

\begin{document}

\maketitle

\begin{abstract}
	For locally convex spaces $X$ and $Y$, the continuous linear map $T:X \to Y$ is called bounded if there is a zero neighborhood $U$ of $X$ such that $T(U)$ is bounded in $Y$. Our main result is that the existence of an unbounded operator $T$ between Fr{\'e}chet spaces $E$ and $F$ which factors through a third Fr{\'e}chet space $G$ ends up with the fact that the triple $(E, G, F)$ has an infinite dimensional closed common nuclear K{\"o}the subspace, provided that $F$ has the property $(y)$.
\end{abstract}
	
\dedicatory{Dedicated to the memory of Prof. Dr. Tosun Terzio{\~g}lu}
\section{Introduction}
	Let $X$ and $Y$ be locally convex spaces. A continuous linear map $T:X \to Y$ is called bounded if there is a $\theta$-neighborhood $U$ of $X$ such that $T(U)$ is bounded in $Y$. We say that a triple $(X,Z,Y)$ has the bounded factorization property and write $(X,Z,Y) \in \mathcal{BF}$ if each linear continuous operator $T : X \to Y$ that factors over $Y$ (that is, $T = R_1 R_2,$ where $R_2 : X \to Z$ and $R_1 : Z \to Y$ are linear continuous operators) is bounded.	Nurlu and Terzio{\~g}lu \cite{Nur84} proved that under some conditions, existence of continuous linear unbounded operators between nuclear K{\"o}the spaces causes existence of common basic subspaces. Djakov and Ramanujan \cite{Dja02-1} sharpened this work by removing nuclearity assumption and using a weaker splitting condition. In \cite{Ter04}, it is shown that the existence of an unbounded factorized operator for a triple of K{\"o}the spaces, under some assumptions, implies the existence of a common basic subspace for at least two of the spaces. Concerning the class of general Fr{\'e}chet spaces, the existence of an unbounded operator inbetween is studied in \cite{Ter86}. It is proved that there is an infinite dimensional closed common nuclear subspace when the range space has a basis, and admits a continuous norm. When the range space has the property $(y)$, that implies the existence of a common nuclear K{\"o}the quotient as proved in \cite{Ter90}. Combining these two results, when the range space has the property $(y)$, common nuclear K{\"o}the subspace is obtained in \cite[Proposition 1]{Ona91}. The aim of this note is to prove the Fr{\'e}chet space analogue of \cite[Proposition 6]{Ter04}, that is, under the condition that $F$ has property $(y)$, and $(E,G,F) \notin \mathcal{BF}$ then there is a common nuclear subspace for all three spaces. We rule out the condition where $G$ can be written as $G=\omega \times X$, where $X$ is a Banach space to avoid the case $T$ becomes almost bounded \cite{Yur93}.
	
	The locally convex space $E$ with neighborhood base $\mathscr{U}(E)$ is said to have property $(y)$ if there is a neighborhood $U_1 \in \mathscr{U}(E)$ such that
	\[
		E' = \bigcup_{U \in \mathscr{U}(E)} \overline{E'[U_1^\circ] \cap U^\circ}.
	\]

\section{Main Result}

\begin{theorem}\label{theorem}
	Let $E,F,G$ be Fr{\'e}chet spaces where $F$ has property $(y)$. Assume there is a continuous, linear, unbounded operator $T:E \to F$ which factors through $G$ such as $T=RS$. Then, there exists an infinite dimensional nuclear K{\"o}the subspace $M$ of $E$ such that the restriction $\displaystyle T|_M$ and the restriction $\displaystyle R|_{S(M)}$ are isomorphisms.
\end{theorem}

\begin{proof}
	Let $T:E \to F$ be an unbounded operator factoring through $G \neq \omega \times X$, for any Banach space $X$.
	
\[
     \xymatrix{
         E \ar[d]_{S}  \ar[r]^{T} & F \\
         G \ar[ur]_{R}   &   }
         \]
         
	By \cite[Proposition 1]{Ona91}, there exists an infinite dimensional closed nuclear K{\"o}the subspace $M$ of $E$ such that the restriction $T|_M$ is an isomorphism onto $T(M)$. Since $T$ is injective on $M$, $R$ is injective on $S(M)$ and maps $S(M)$ onto $T(M)=R(S(M))$. Using the fact that $T$ is an isomorphism, it is easy to verify that $R|_{S(M)}$ is one-to-one. Now let $y \in \overline{S(M)}$. So find a sequence $(S(m_n))_{n \in \mathbb{N}}$ in $S(M)$ such taht $\lim S(m_n)=y$. $R$ is continuous at y, then $RS(m_n)=T(m_n)=Ry \in \overline{T(M)}=T(M)$, since $T(M)$ is closed. Thus $\lim T(m_n)=T(m)=Ry$ for some $m \in M$. Since $T$ is an isomorphism on $M$, $\lim T^{-1}T(m_n)=T^{-1}T(m)$, that is, $\lim m_n = m$. $S$ is continuous at $m$, and that implies $\lim S(m_n)=S(m)=y \in S(M)$. Therefore $S(M)$ is closed. Hence $R:S(M) \to R(S(M))$ is an isomorphism by the Open Mapping Theorem.
\end{proof}

As proved in \cite[Lemma 2.1]{Vog91} and \cite[Theorem 2.3]{Vog91}, property $(y)$, which is assumed to be enjoyed by $F$ can be replaced by being locally closed, or being isomorphic to a closed subspace of a K{\"o}the space. It is shown that these conditions are equivalent to have the property $(y)$.
	
\bibliography{factorization}
\bibliographystyle{siam}
\end{document}